\documentclass[11 pt,dvipsnames]{article}
\usepackage{graphicx}
\usepackage{camnum}
\usepackage{amsmath}
\usepackage{amsfonts}
\usepackage{epstopdf}
\usepackage{xcolor}
\usepackage{harvard}

\numberwithin{figure}{section}
\numberwithin{equation}{section}
\numberwithin{table}{section}
\numberwithin{theorem}{section}

\allowdisplaybreaks

\newcommand{\DDD}{\mathcal{D}}
\newcommand{\DDT}{\widetilde{\mathcal{D}}}

\newcommand{\G}{\CC{\Gamma}}
\newcommand{\ee}{{\mathrm e}}
\newcommand{\ii}{{\mathrm i}}

\newcommand{\hyper}[5]
       {{}_{#1} F_{#2} \!\left[
           \begin{array}{l}
         #3;\\#4;
           \end{array}#5\right]
       }

\newcommand{\PP}{\mathrm{P}}

\date{}

\title{On semi-separability and differentiation matrices}

\begin{document}
\thispagestyle{empty}

\author{Arieh Iserles\\Department of Applied Mathematics and Theoretical Physics\\University of Cambridge}

\maketitle

\begin{abstract}
  The theory of spectral methods for partial differential equations leads to infinite-dimensional matrices which represent the derivative operator with respect to an underlying orthonormal basis. Favourable properties of such {\em differentiation matrices\/} are crucial in the design of good spectral methods. It is known that bases using Laguerre and ultraspherical polynomials lead to semi-separable differentiation matrices of rank~1. In this paper we consider orthonormal bases constructed from Jacobi polynomials and prove that the underlying differentiation matrices are semi-separable of rank~2. This requires new results on semi-separable matrices which might be interesting in a wider context.
\end{abstract}

\tableofcontents

\section{Introduction}

This paper is motivated by spectral methods for time-dependent partial differential equations, aims at resolving a problem concerning matrix analysis and numerical linear algebra and uses methodology of the theory of orthogonal polynomials.

\subsection{Spectral methods}

In order to motivate the narrative of this paper, let us consider a time-dependent partial differential equation in a single space variable,
\begin{equation}
  \label{eq:1.1}
  \frac{\partial u}{\partial t}=\mathcal{K}(t,x,u),\qquad x\in(a,b),\quad t\geq0,
\end{equation}
given with an initial condition $u(x,0)=u_0(x)$, $x\in\mathcal{H}$, where $\mathcal{H}$ is a separable Hilbert space in $(a,b)$, with appropriate zero boundary conditions at $a$ and $b$. $\mathcal{K}$ is a differential operator, which need not be linear, of appropriate smoothness and we assume that \R{eq:1.1} is well posed in $\mathcal{H}$. 

While there are many ways to define a spectral method for the numerical solution of \R{eq:1.1} \cite{canuto06sm,gottlieb77nas,hesthaven07std,trefethen00smm}, we adopt in this paper a minimalist approach: {\em a spectral method is an orthonormal basis $\Phi=\{\varphi_n\}_{n\in\mathbb{Z}_+}$ of $\mathcal{H}$.\/} The main idea is to expand all the relevant quantities, inclusive of the solution, in the basis $\Phi$ -- in other words, a function $f\in\mathcal{H}$ is associated with a sequence $\hat{\MM{f}}\in\ell_2[\mathbb{Z}_+]$ such that 
\begin{displaymath}
  f(x)=\sum_{n=0}^\infty \hat{f}_n \varphi_n(x),\qquad x\in(a,b)
\end{displaymath}
in the weak-* topology of $\mathcal{H}$. (And the solution $u(x,t)$ of \R{eq:1.1} with the sum $\sum_{n=0}^\infty \hat{u}_n(t)\varphi_n(x)$.) Let $\mathcal{N}:\mathcal{H}\rightarrow\ell_2(\mathbb{Z}_+)$ be the {\em canonical map\/} taking $f$ to $\hat{\MM{f}}$: it follows from the Plancherel theorem that $\mathcal{N}$ is an isometric bijection and $\mathcal{N}^{-1}=\mathcal{N}^*$ where $\mathcal{N}^*$ is the adjoint. Therefore, in place of \R{eq:1.1}, we obtain a partial differential equation evolving in $\ell_2[\mathbb{Z}_+]$, namely
\begin{equation}
  \label{eq:1.2}
  \frac{\partial \hat{\MM{u}}}{\partial t}=\mathcal{N}{\mathcal K}(t,x,\mathcal{N}^*\hat{\MM{u}}),\qquad x\in(a,b),\quad t\geq0,
\end{equation}
with the initial condition $\hat{\MM{u}}_0=\mathcal{N}u_0$ and zero boundary conditions. We turn \R{eq:1.2} into a spectral method by restricting the underlying sequences to a finite-dimensional space, whereby standard Galerkin conditions result in a (finite-dimensional) system of ordinary differential equations (ODEs). The latter can be solved by a wide range of methods, some standard and other optimised to this setting.

A design of a good spectral method rests upon several imperatives, which are often in tension:
\begin{enumerate}
\item {\bf Stability:} The finite-dimensional ODE system derived from \R{eq:1.2} must be stable \cite{iserles09fna}. In a non-technical language, this means that  it is uniformly well posed as the  number of degrees of freedom becomes infinite, thereby maintaining the well-posedness of \R{eq:1.1}. Note that this is not an optional extra but a necessary condition for the method being viable as an approximation to the exact solution. 
\item {\bf Preservation of structure:} The equation \R{eq:1.1} often possesses structural features whose preservation under discretisation is essential for underlying applications. For example, the solution of the {\em linear Schr\"odinger equation\/} $\ii \partial u/\partial t=-\partial^2u/\partial x^2+V(x)u$, where $V$ is real, preserves the $\CC{L}_2$ norm (subject to periodic or Cauchy boundary conditions) and this is crucial to the quantum-mechanical interpretation of $|u(x,t)|^2$ as a probability distribution. Another example is the {\em diffusion equation\/} $\partial u/\partial t=\partial [a(x)\partial u/\partial x]/\partial x$, where $a>0$ and $u_0\geq0$, whose solution preserves non-negativity. Ideally, we wish to conserve such features under discretisation \cite{hairer06gni}.
\item {\bf The computation of the canonical map:} The implementation of a spectral methods requires, often repeatedly, to apply the canonical map $\mathcal{N}$, in other words to expand in the basis $\Phi$. While this is not the subject matter of this paper, we mention the existence of fast algorithms to this end \cite{olver20fau}.
\item {\bf Fast convergence:} The speed of convergence of an expansion in the basis $\Phi$ is crucial to the success of a spectral method, not least because the smaller the eventual system, the faster the numerical linear algebra. We will not pursue this issue further in this paper, while mentioning that the speed of convergence depends strongly on the nature of the Hilbert space $\mathcal{H}$ and on the function being approximated: clearly, we can expect much faster convergence for an analytic function, say, than for a function which is merely at $\CC{L}_2(a,b)$.
\item {\bf Numerical linear algebra:} Ultimately, the purpose of the exercise is to reduce analysis to linear algebra. The choice of the basis $\Phi$ has fundamental impact on the nature of the algebraic systems that occur once the ODE that we have derived from \R{eq:1.2} through Galerkin conditions is being computed. A well-known case is that of periodic boundary conditions, whereby the choice of a Fourier basis leads to fast numerical linear algebra using the Fast Fourier Transform. At its core, the purpose of this paper is to construct a framework for fast algebra for spectral methods in the presence of zero boundary conditions in a compact interval.
\end{enumerate}

In this paper our framework is restricted to univariate (in space) problems defined in compact intervals and to zero boundary conditions. The formalism of W-functions, adopted in the sequel as the organising principle of this paper, can be easily extended to parallelepipeds {\em via\/} tensor products, but also to simplexes \cite{gao25ima}. Cauchy problems call for different formalism, cf.\ for example \cite{iserles19oss}. Insofar as boundary conditions are concerned, in a periodic setting (and, more generally, on a torus) nothing beats Fourier expansions, while nonzero Dirichlet boundary conditions using W-functions require at the present state of knowledge an {\em ad hoc\/} reduction to zero boundary conditions \cite{gao25ima}. We examine solely the case of $\mathcal{H}=\CC{L}_2(a,b)$, both because its theory is more mature and since it is the right setting for a wide variety of problems: the diffusion equation dissipates in $\CC{L}_2$ while dispersive equations (e.g.\ linear and nonlinear Schr\"odinger, Korteweg--de Vries, Klein--Gordon and Dirac equations) conserve $\CC{L}_2$ (with appropriate boundary conditions).

\subsection{Differentiation matrices and W-functions}

Given any $\CC{C}^1$ orthonormal basis of the Hilbert space $\mathcal{H}$, there exists a linear map $\DDD$ taking $\Phi$ to $\Phi'=\{\varphi_n'\}_{n\in\mathbb{Z}_+}$,
\begin{displaymath}
  \varphi_m'=\sum_{n=0}^\infty \DDD_{m,n}\varphi_n,\quad m\in\mathbb{Z}_+,\qquad\mbox{where}\qquad \DDD_{m,n}=\langle \varphi_m',\varphi_n\rangle.
\end{displaymath}
The matrix $\DDD$ is called the {\em differentiation matrix\/} of $\Phi$ and it plays in the realm of $\ell_2[\mathbb{Z}_+]$ a role parallel to a derivative in $\mathcal{H}$,
\begin{displaymath}
  f=\sum_{n=0}^\infty \hat{f}_n\varphi_n\qquad\Rightarrow\qquad f'=\sum_{n=0}^\infty (\DDD\hat{\MM{f}})_n \varphi_n
\end{displaymath}
and similarly for higher derivatives. (In several dimensions we need `partial' differentiation matrices, cf.\ \cite{gao25ffs,gao25ima}.) Bearing in mind that our purpose is to approximate differential equations, hence derivatives, little surprise that a differentiation matrix is central to this endeavour. 

The derivative operator, accompanied by zero Dirichlet, periodic or Cauchy boundary conditions, is skew Hermitian with respect to the standard $\CC{L}_2$ inner product, $\langle f',g\rangle+\langle f,g'\rangle=0$, and it is highly beneficial for a differentiation matrix to inherit this feature, namely $\DDD+\DDD^*=O$. This is the case with the Fourier basis in the presence of periodic boundary conditions, while for the Cauchy problem all orthonormal systems with a tridiagonal, skew-Hermitian differentiation matrix have been characterised in \cite{iserles19oss}. Once $\DDD$ is skew-Hermitian, proofs of stability become much easier and often the method can be guaranteed to respect important invariants and geometric features.

Let $w$ be a weight function (a nonzero, nonnegative, integrable function with all its moments bounded) supported by the interval $(a,b)$. It thus generates an inner product $\langle f,g\rangle=\int_a^b f(x)g(x)w(x)\D x$ and a basis of orthonormal polynomials $\{p_n\}_{n\in\mathbb{Z}_+}$. We henceforth stipulate that $\mathcal{H}=\CC{L}_2(a,b)$ and follow \citeasnoun{iserles24ssm}, defining a set of so-called {\em W-functions\/} $\Phi=\{\varphi_n\}_{n\in\mathbb{Z}_+}$ by letting
\begin{equation}
  \label{eq:1.3}
  \varphi_n(x)=\sqrt{w(x)} p_n(x),\qquad n\in\mathbb{Z}_+.
\end{equation}
To impose zero boundary conditions we further stipulate that $w(a)=w(b)=0$, whereby $\Phi$ becomes a basis of $\CC{L}_2(a,b)$ with zero boundary conditions. It has been proved in \cite{iserles24ssm} that, subject to the above condition on the weight function, for $w\in\CC{C}^1$ the matrix $\DDD$ is skew symmetric and
\begin{equation}
  \label{eq:1.4}
  \DDD_{m,n}=\begin{case}
      \displaystyle\tfrac12 \int_a^b w'(x)p_m(x)p_n(x)\D x, & m\leq n-1,\\[8pt]
      0, & m=n,\\[4pt]
      \displaystyle-\tfrac12 \int_a^b w'(x)p_m(x)p_n(x)\D x, & m\geq n+1,
   \end{case} \qquad m,n\in\mathbb{Z}_+.
\end{equation}

\subsection{Semi-separability}

A familiar argument counterposes finite element and finite difference methods {\em vs\/} spectral methods: supposedly, finite element and finite difference methods lead to algebraic linear systems which, while very large, are also  sparse, lending themselves to fast linear algebra, while spectral methods result in algebraic systems which are much smaller, yet dense. Inasmuch as this is not the only major difference between these methods, it definitely looms large in many decisions on the `right' approach to a given problem. This argument is false.

The structure of the algebraic system that needs to be solved to time-step a truncated system \R{eq:1.2} depends in large measure upon the structure and sparsity of the differentiation matrix. 

In the presence of periodic boundary conditions a Fourier basis leads to a diagonal $\DDD$, while for Cauchy boundary conditions we can choose $\Phi$ that gives rise to a tridiagonal matrix: matrices of this kind that allow for an $N\times N$ truncation of the canonical map $\mathcal{N}$ to be computed in $\O{N\log N}$ operations have been characterised in \cite{iserles21fco}. However, once we consider zero Dirichlet boundary conditions, in general the matrix $\DDD$ in \R{eq:1.4} is dense. However, in the two cases considered in \cite{iserles24ssm}, namely $\mathcal{H}=\CC{L}_2(-1,1)$ with the {\em ultraspherical basis\/} $\{\PP_n^{(\alpha,\alpha)}\}_{n\in\mathbb{Z}_+}$, $\alpha>0$,\footnote{These polynomials need to be normalised to form an orthonormal basis.} and $\mathcal{H}=\CC{L}_2(0,\infty)$ with the {\em Laguerre basis\/} $\{\CC{L}_n^{(\alpha)}\}_{n\in\mathbb{Z}_+}$, $\alpha>0$,\footnote{Ditto.} the differentiation matrices $\DDD$ are known explicitly and they are semi-separable of rank~1.

We recall that a matrix is semi-separable of rank $r\geq1$ if any sub-matrix wholly above or wholly below the main diagonal is of rank $\leq r$ \cite{vandebril08mcs}. Once a matrix $\mathcal{A}$ is semi-separable of rank $r$, an $N\times N$ linear system of the form $\mathcal{A}\MM{v}=\MM{b}$ can be reduced by elementary operations to a banded matrix with $2r+1$ diagonals and  solved in $\O{N}$ operations. Likewise it takes $\O{N}$ operations to multiply such a matrix by a vector -- in other words, compute an expansion of a derivative from an expansion of a function in our setting. This makes for fast algebra!

It has been proved in \cite{iserles24ssm} that the Laguerre basis leads to a semi-separable differentiation matrix of rank~1. For ultraspherical basis we have $\DDD_{m,n}=0$ whenever $m+n$ is even, hence the task (e.g.\ computing a linear system) can be reduced to two problems, each of half length. These problems have matrices which are semi-separable of rank 1.

\subsection{A plan of this paper}

The main purpose of this paper is to consider the {\em Jacobi weight\/} $w(x)=(1-x)^\alpha(1+x)^\beta$, $x\in(-1,1)$, where, to ensure $w(\pm1)=0$, we take $\alpha,\beta>0$. (As a matter of fact, it is advantageous to take both $\alpha$ and $\beta$ as even positive integers, whereby the $\varphi_n$s are analytic and convergence is considerably more rapid \cite{iserles24ssm}, but this plays no further role in the current work.) This is for two reasons. The first is practical: to cater for the case when the number of zero boundary conditions from the left and the right is different. If, say, we wish to impose the boundary conditions $u(-1,t)=u(1,t)=\partial u(1,t)/\partial x=0$, the right choice is  $\alpha=4$, $\beta=2$: this means that the $\varphi_n$s are analytic and $\sqrt{w}$ vanishes to first order at $x=-1$ and second order at $x=+1$ . The second reason is theoretical, to explore further the structure of differentiation matrices for different choices of W-functions. 

In Section~2 we employ the theory of orthogonal polynomials to derive $\DDD$ explicitly for the Jacobi weight. This is used in Section~3 to prove that the differentiation matrix is semi-separable of rank 2. As a matter of fact, we accomplish much more, proving that (subject to a boundedness condition for infinite-dimensional matrices) a product of a semi-separable matrix of rank $s$ by a semi-separable matrix of rank $r$ is semi-separable of rank $\leq r+s$. This is of practical interest because, implementing spectral methods,   we often need to form products of differentiation matrices, e.g.\ to represent the Laplacian. Finally, in Section~4, we consider two speculative issues that follow from our analysis and might foster future research.

\section{An explicit form of the differentiation matrix}

It is convenient to list at the outset  all  facts about Jacobi polynomials that will be needed in the sequel. Thus, recall that the Jacobi polynomials $\PP_n^{(\alpha,\beta)}$ are orthogonal in $(-1,1)$ with respect to the measure $w_{\alpha,\beta}(x)=(1-x)^\alpha(1+x)^\beta$, where $\alpha,\beta>-1$,
\begin{equation}
  \label{eq:2.1}
  \int_{-1}^1 (1-x)^\alpha(1+x)^\beta \PP_m^{(\alpha,\beta)}(x)\PP_n^{(\alpha,\beta)}(x)\D x=\delta_{m,n} \kappa_n^{-2},
\end{equation}
where
\begin{equation}
  \label{eq:2.2}
  \kappa_n=\sqrt{\frac{(2n+\alpha+\beta+1)\G(n+\alpha+\beta+1)}{2^{\alpha+\beta+1}\G(n+\alpha+1)\G(n+\beta+1)}},\qquad m,n\in\mathbb{Z}_+
\end{equation}
\cite[Table 18.3.1]{dlmf}. Here $\delta_{m,n}$ is the delta of Kronecker: $\delta_{m,m}=1$ and $\delta_{m,n}=0$ for $m\neq n$. Jacobi polynomials obey a three-term recurrence relation, which we write in the form
\begin{Eqnarray}
  \nonumber
  &&2(n+1)(n+\alpha+\beta+1)(2n+\alpha+\beta)\PP_{n+1}^{(\alpha,\beta)}\\
  \nonumber
  &&\mbox{}-(\alpha^2-\beta^2)(2n+\alpha+\beta+1)\PP_n^{(\alpha,\beta)}\\
  \label{eq:2.3}
  &&\mbox{}+2(n+\alpha)(n+\beta)(2n+\alpha+\beta+2)\PP_{n-1}^{(\alpha,\beta)}\\
  \nonumber
  &=&(2n+\alpha+\beta)(2n+\alpha+\beta+1)(2n+\alpha+\beta+2)x\PP_n^{(\alpha,\beta)}
\end{Eqnarray}
\cite[Eqn 18.9.2]{dlmf}. (Here and elsewhere $\PP_{-1}^{(\alpha,\beta)}\equiv0$.) The underlying orthonormal Jacobi polynomials are $p_n=\kappa_n \PP_n^{(\alpha,\beta)}$, therefore the W-functions are
\begin{equation}
  \label{eq:2.4}
  \varphi_n(x)=(1-x)^{\alpha/2}(1+x)^{\beta/2}p_n(x)=\kappa_n (1-x)^{\alpha/2}(1+x)^{\beta/2}\PP_n^{(\alpha,\beta)}(x),\qquad n\in\mathbb{Z}_+
\end{equation}
-- recall that we need to restrict the range of parameters to $\alpha,\beta>0$ to ensure  skew-symmetry of the differentiation matrix -- while the differentiation matrix is
\begin{displaymath}
  \DDD_{m,n}=\int_{-1}^1 \varphi_m'(x)\varphi_n(x)\D x,\qquad m,n\in\mathbb{Z}_+.
\end{displaymath}
It is helpful to define by this stage the {\em pre-W-functions\/}
\begin{displaymath}
  \tilde{\varphi}_n(x)= (1-x)^{\alpha/2}(1+x)^{\beta/2}\PP_n^{(\alpha,\beta)}(x),\qquad n\in\mathbb{Z}_+
\end{displaymath}
and the {\em pre-differentiation matrix\/}
\begin{displaymath}
  \DDT_{m,n}=\int_{-1}^1 \tilde{\varphi}_m'(x)\tilde{\varphi}_n(x)\D x,\qquad m,n\in\mathbb{Z}_+,
\end{displaymath}
noting that 
\begin{equation}
  \label{eq:2.5}
  \DDD_{m,n}=\kappa_m\kappa_n \DDT_{m,n},\qquad m,n\in\mathbb{Z}_+.
\end{equation}

Finally, we use \cite[18.9.5]{dlmf} and $\PP_n^{(\alpha,\beta)}(x)=(-1)^n\PP_n^{(\beta,\alpha)}(-x)$ \cite[p.~256]{rainville60sf} to claim that
\begin{Eqnarray}
  \label{eq:2.6}
  (2n+\alpha+\beta)\PP_n^{(\alpha,\beta-1)}(x)&=&(n+\alpha+\beta)\PP_n^{(\alpha,\beta)}(x)+(n+\alpha)\PP_{n-1}^{(\alpha,\beta)}(x),\\
  \nonumber
  (2n+\alpha+\beta)\PP_n^{(\alpha-1,\beta)}(x)&=&(n+\alpha+\beta)\PP_n^{(\alpha,\beta)}(x)-(n+\beta)\PP_{n-1}^{(\alpha,\beta)}(x),\quad n\in\mathbb{Z}_+.
\end{Eqnarray}

We will derive the pre-differentiation matrix first and then use \R{eq:2.5} to deduce the explicit form of the differentiation matrix. We need to concern ourselves just with $m\geq n+1$, because the case $m\leq n$ follows by skew-symmetric completion. 

Our first task is to determine $\DDT_{m,0}$ for $m\in\mathbb{N}$. Recalling \R{eq:1.4}, we have
\begin{displaymath}
  \DDT_{m,n}=\tfrac12 \int_{-1}^1 w_{\alpha,\beta}'(x)\PP_m^{(\alpha,\beta)}(x)\PP_n^{(\alpha,\beta)}(x)\D x,\qquad m\leq n-1.
\end{displaymath}
However,
\begin{displaymath}
  w'_{\alpha,\beta}(x)=-\alpha w_{\alpha-1,\beta}(x)+\beta w_{\alpha,\beta-1}(x),
\end{displaymath}
and $\PP_0^{(\alpha,\beta)}\equiv1$, therefore, integrating by parts and exploiting $w_{\alpha,\beta}(\pm1)=0$,
\begin{Eqnarray}
  \nonumber
  \DDT_{m,0}&=&\int_{-1}^1 \frac{\D}{\D x}\!\left[w^{\frac12}_{\alpha,\beta}(x)\PP_m^{(\alpha,\beta)}(x)\right]\! w^{\frac12}_{\alpha,\beta}(x)\D x=-\frac12 \int_{-1}^1 w'_{\alpha,\beta} \PP_m^{(\alpha,\beta)}(x)\D x\\
  \nonumber
  &=&\frac{\alpha}{2}\int_{-1}^1 w_{\alpha-1,\beta}(x) \PP_m^{(\alpha,\beta)}(x)\D x-\frac{\beta}{2}\int_{-1}^1 w_{\alpha,\beta-1}(x)\PP_m^{(\alpha,\beta)}(x)\D x\\
  \label{eq:2.7}
  &=&\frac{\alpha}{2} \GG{t}_m^{(\alpha,\beta)}-\frac{\beta}{2}\GG{s}_m^{(\alpha,\beta)}
\end{Eqnarray}
where
\begin{displaymath}
  \GG{t}_m^{(\alpha,\beta)}=\int_{-1}^1 w_{\alpha-1,\beta} \PP_m^{(\alpha,\beta)}(x)\D x,\quad \GG{s}_m^{(\alpha,\beta)}=\int_{-1}^1 w_{\alpha,\beta-1}\PP_m^{(\alpha,\beta)}(x)\D x,\qquad\! n\in\mathbb{Z}_+.
\end{displaymath}

\begin{lemma}
  It is true that
  \begin{equation}
    \label{eq:2.8}
    \GG{t}_m^{(\alpha,\beta)}=2^{\alpha+\beta}\frac{\G(\alpha)\G(m+\beta+1)}{\G(m+\alpha+\beta+1)},\quad \GG{s}_m^{(\alpha,\beta)}=(-1)^m 2^{\alpha+\beta} \frac{\G(m+\alpha+1)\G(\beta)}{\G(m+\alpha+\beta+1)}
  \end{equation}
  for all $m\in\mathbb{Z}_+$.
\end{lemma}

\begin{proof}
  We use \R{eq:2.6} and orthogonality \R{eq:2.1} to deduce that
  \begin{Eqnarray*}
    &&(m+\alpha+\beta)\int_{-1}^1 w_{\alpha-1,\beta}(x)\PP_m^{(\alpha,\beta)}(x)\D x-(m+\beta)\int_{-1}^1 w_{\alpha-1,\beta}(x)\PP_{m-1}^{(\alpha,\beta)}(x)\D x\\
    &&=(2m+\alpha+\beta)\int_{-1}^1 w_{\alpha-1,\beta}\PP_m^{(\alpha-1,\beta)}(x)\D x=0,
  \end{Eqnarray*}
  therefore
  \begin{displaymath}
    (m+\alpha+\beta)\GG{t}_m^{(\alpha,\beta)}=(m+\beta)\GG{t}_{m-1}^{(\alpha,\beta)}.
  \end{displaymath}
  Since, by direct calculation,
  \begin{displaymath}
    \GG{t}_0^{(\alpha,\beta)}=2^{\alpha+\beta} \frac{\G(\alpha)\G(\beta+1)}{\G(\alpha+\beta+1)},
  \end{displaymath}
  the expression for $\GG{t}_m^{(\alpha,\beta)}$ follows by induction.
  
  The result for $\GG{s}_m^{(\alpha,\beta)}$ now follows because $\PP_m^{(\alpha,\beta)}(x)=(-1)^m\PP_m^{(\beta,\alpha)}(-x)$, while replacing the first identity in \R{eq:2.6} by the second.
\end{proof}

Recall that the {\em Pochhammer symbol\/} is defined for any $z\in\mathbb{C}$ and $m\in\mathbb{Z}_+$ by
\begin{displaymath}
  (z)_m=\prod_{k=0}^{m-1} (z+k)=\frac{\G(z+m)}{\G(z)}
\end{displaymath}
(the rightmost equality makes sense only once $z$ is neither zero nor a negative integer). It now follows from \R{eq:2.8} that
\begin{equation}
  \label{eq:2.9}
  \DDT_{m,0}=2^{\alpha+\beta-1}\frac{\G(\alpha+1)\G(\beta+1)}{\G(\alpha+\beta+m+1)}[(\beta+1)_m- (-1)^m(\alpha+1)_m],\qquad m\in\mathbb{Z}_+.
\end{equation}

We have the beginning of a thread that will lead us to all $\DDT_{m,n}$s -- at the first instance for $m\geq n+1$ and then, by skew symmetry, to the remainder of the pre-differentiation matrix. We proceed similarly to \cite{gao23rrg}, commencing from the three-term recurrence \R{eq:2.3}. Multiplying by $w_{\alpha,\beta}'\PP_m^{(\alpha,\beta)}/2$ and integrating in $(-1,1)$, recalling that $m\geq n+1$ we obtain
\begin{Eqnarray*}
  &&-\tfrac12 (2n+\alpha+\beta)(2n+\alpha+\beta+1)(2n+\alpha+\beta+2) \int_{-1}^1 w'_{\alpha,\beta}x\PP_m^{(\alpha,\beta)}\PP_n^{(\alpha,\beta)}\D x\\
  &=&2(n+1)(n+\alpha+\beta+1)(2n+\alpha+\beta)\DDT_{m,n+1}\\
  &&\mbox{}-(\alpha^2-\beta^2)(2n+\alpha+\beta+1)\DDT_{m,n}\\
  &&\mbox{}+2(n+\alpha)(n+\beta)(2n+\alpha+\beta+2)\DDT_{m,n-1}.
\end{Eqnarray*}
Likewise, by symmetry,
\begin{Eqnarray*}
  &&-\tfrac12 (2m+\alpha+\beta)(2m+\alpha+\beta+1)(2m+\alpha+\beta+2) \!\int_{-1}^1\! w'_{\alpha,\beta}x\PP_m^{(\alpha,\beta)}\PP_n^{(\alpha,\beta)}\D x\!\\
  &=&2(m+1)(m+\alpha+\beta+1)(2m+\alpha+\beta)\DDT_{m+1,n}\\
  &&\mbox{}-(\alpha^2-\beta^2)(2m+\alpha+\beta+1)\DDT_{m,n}\\
  &&\mbox{}+2(m+\alpha)(m+\beta)(2m+\alpha+\beta+2)\DDT_{m-1,n}.
\end{Eqnarray*}
Comparing $\int_{-1}^1 xw'_{\alpha,\beta}(x)\PP_m^{(\alpha,\beta)}(x)\PP_n^{(\alpha,\beta)}(x)\D x$ in both expressions, we deduce the bilateral recurrence relation
\begin{equation}
  \label{eq:2.10}
  c_n\DDT_{m,n-1}+d_n\DDT_{m,n}+e_n\DDT_{m,n+1}=c_m\DDD_{m-1,n}+d_m\DDT_{m,n}+e_m\DDT_{m+1,n},
\end{equation}
where
\begin{Eqnarray*}
  c_n&=&\frac{(n+\alpha)(n+\beta)}{(2n+\alpha+\beta)(2n+\alpha+\beta+1)},\\
  d_n&=&-\frac{\alpha^2-\beta^2}{2} \frac{1}{(2n+\alpha+\beta)(2n+\alpha+\beta+2)},\\
  e_n&=&\frac{(n+1)(n+\alpha+\beta+1)}{(2n+\alpha+\beta+1)(2n+\alpha+\beta+2)}.
\end{Eqnarray*}

\begin{theorem}
  \label{th:2.2}
  The differentiation matrix for the Jacobi W-functions is
  \begin{Eqnarray}
    \label{eq:2.11}
    &&\DDD_{m,n}=\frac{1}{4m!} \sqrt{\frac{\G(m+\alpha+\beta+1)}{\G(n+\alpha+\beta+1)} (2n+\alpha+\beta+1)(2m+\alpha+\beta+1)}\hspace*{40pt}\\
    \nonumber
    &&\mbox{}\times\!\left[\sqrt{\frac{\G(m\!+\!\alpha\!+\!1)\G(n\!+\!\beta\!+\!1)}{\G(n\!+\!\alpha\!+\!1)\G(m\!+\!\beta\!+\!1)}}-(-1)^{n-m}\sqrt{\frac{\G(n\!+\!\alpha\!+\!1)\G(m\!+\!\beta\!+\!1)}{\G(m\!+\!\alpha\!+\!1)\G(n\!+\!\beta\!+\!1)}}\right]
  \end{Eqnarray}
  for $m,n\in\mathbb{Z}_+$, $m\leq n-1$.
\end{theorem}

\begin{proof}
  We use the recurrence \R{eq:2.10} for $m\geq n+1$ with the boundary conditions $\DDT_{m,-1}=0$, $\DDT_{m,0}$ given by \R{eq:2.9} (and which is consistent with \R{eq:2.11}) and $\DDT_{m,m}=0$. Direct substitution, accompanied by simple algebraic manipulation, yields
  \begin{Eqnarray*}
     \DDT_{m,n}&=&2^{\alpha+\beta-1} \frac{\G(m+\alpha+1)\G(m+\beta+1)}{m!\G(n+\alpha+\beta+1)} [(m+\beta+1)_{n-m}\\
     &&\hspace*{20pt}\mbox{}-(-1)^{n-m}(m+\alpha+1)_{n-m}]\\
     &=&2^{\alpha+\beta-1}\frac{\G(m\!+\!\alpha\!+\!1)\G(n\!+\!\beta\!+\!1)}{m!\G(n+\alpha+\beta+1)-(-1)^{n-m}\G(n\!+\!\alpha\!+\!1)\G(m\!+\!\beta\!+\!1)} 
   \end{Eqnarray*}
   for all $m\geq n+1$. To conclude the proof we use \R{eq:2.5}, followed by further straightforward algebra.
\end{proof}

Observe that for $\alpha=\beta$ we have $\DDT_{m,n}=0$ whenever $m+n$ is even, as we already know from \cite{iserles24ssm}. 

Once $\alpha,\beta\in\mathbb{N}$ \R{eq:2.11} simplifies to
\begin{Eqnarray*}
  \DDD_{m,n}&=&\frac{1}{4m!}\sqrt{\frac{(m+\alpha+\beta)!(2n+\alpha+\beta+1)(2m+\alpha+\beta+1)}{(n+\alpha+\beta)!}}\\
  &&\mbox{}\times\!\left[\sqrt{\frac{(m+\alpha)!(n+\beta)!}{(n+\alpha)!(m+\beta)!}}-(-1)^{n-m}\sqrt{\frac{(n+\alpha)!(m+\beta)!}{(m+\alpha)!(n+\beta)!}}\right]
\end{Eqnarray*}
for $m\leq n-1$, with skew-symmetric completion, by virtue of the identity $\G(\gamma+1)=\gamma!$ for $\gamma\in\mathbb{N}$.

\section{Semi-separability}

It is elementary that the matrix $A$ is semi-separable of rank $r$  only if there exist vectors $\GG{a}^{[i]},\GG{b}^{[i]},\GG{c},\GG{d}^{[i]},\GG{e}^{[i]}$, $i=1,\ldots,r$ of commensurate length (in our case, in $\ell_\infty[\mathbb{Z}_+]$) such that
\begin{equation}
  \label{eq:3.1}
  A_{m,n}=\begin{case}
     \displaystyle \sum_{i=1}^r \GG{a}_m^{[i]}\GG{b}_n^{[i]}, & m\leq n-1,\\
     \GG{c}_n, & m=n,\\
     \displaystyle \sum_{i=1}^r \GG{d}_m^{[i]}\GG{e}_n^{[i]}, & m\geq n+1
  \end{case}
\end{equation}
and the matrices $[\GG{a}^{[1]},\ldots,\GG{a}^{[r]}]$, $[\GG{b}^{[1]},\ldots,\GG{b}^{[r]}]$, $[\GG{d}^{[1]},\ldots,\GG{d}^{[r]}]$ and $[\GG{e}^{[1]},\ldots,\GG{e}^{[r]}]$ are of full rank. This follows at once from \cite{vandebril08mcs}. Further simplifications  occur once $A$ is skew-symmetric, whereby it is enough to specify just the vectors $\GG{a}^{[i]}$ and $\GG{b}^{[i]}$, while $\GG{c}=\MM{0}$, $\GG{d}^{[i]}=\GG{b}^{[i]}$ and $\GG{e}^{[i]}=-\GG{a}^{[i]}$. 

\begin{theorem}
  \label{thm:3.1}
  The differentiation matrix $\DDD$ corresponding to the Jacobi measure with $\alpha,\beta>0$ is semi-separable of rank~2 with
  \begin{Eqnarray*}
  \GG{a}_m^{[1]}&=&\frac{(-1)^{m-1}}{2m!} \sqrt{\frac{(2m+\alpha+\beta+1)\G(m+\beta+1)\G(m+\alpha+\beta+1)}{\G(m+\alpha+1)}},\\
  \GG{b}_n^{[1]}&=&\frac{(-1)^n}{2}\sqrt{\frac{(2n+\alpha+\beta+1)\G(n+\alpha+1)}{\G(n+\beta+1)\G(n+\alpha+\beta+1)}},\\
  \GG{a}_m^{[2]}&=&\frac{1}{2m!} \sqrt{\frac{(2m+\alpha+\beta+1)\G(m+\alpha+1)\G(m+\alpha+\beta+1)}{\G(m+\beta+1)}},\\
  \GG{b}_n^{[2]}&=&\frac12 \sqrt{\frac{(2n+\alpha+\beta+1)\G(n+\beta+1)}{\G(n+\alpha+1)\G(n+\alpha+\beta+1)}}.
\end{Eqnarray*}
\end{theorem}

\begin{proof}
  Follows at once from Theorem~2.2 by inspection.
\end{proof}

Going back to the motivation of this paper, spectral methods for time-dependent PDEs, we wish to consider not just the differentiation matrix but also its square: if $\DDD$ corresponds to $\partial/\partial x$ in the `parameter space', $\DDD^2$ corresponds to the second derivative -- in a multivariate setting to the Laplacian. While the first derivative (or a gradient) are skew-Hermitian in an appropriate Hilbert space $\mathcal{H}$, the Laplacian is negative definite and, for all the reasons elaborated upon in Section~1, it is a most welcome news that skew-Hermitian $\DDD$ implies that $\DDD^2$ is negative semidefinite.\footnote{In general, $\DDD^{2n+1}$ is skew-symmetric, $\DDD^{2n+2}$ negative semidefinite and $\DDD^{4n}$ positive semi-definite for all $n\in\mathbb{Z}_+$.} But is it semi-separable and, if so, of which rank?

Before we address this question, we need to explore further general properties of semi-separable matrices.

\begin{lemma}
  \label{lemma:3.2}
  Let $\mathcal{A}$ and $\mathcal{B}$ be two skew-symmetric semi-separable matrices of rank 1,
  \begin{equation}
    \label{eq:3.2}
    \mathcal{A}_{m,n}=\begin{case}
      \GG{a}_m\GG{b}_n, & m\leq n-1,\\[5pt]
      \GG{f}_n, & m=n,\\[5pt]
      \GG{d}_m\GG{e}_n, & m\geq n+1,
    \end{case}\qquad 
    \mathcal{B}_{m,n}=\begin{case}
      \GG{p}_m\GG{q}_n, & m\leq n-1,\\[5pt]
      \GG{r}_n, & m=n,\\[5pt]
      \GG{s}_m\GG{t}_n, & m\geq n+1,
    \end{case}
  \end{equation}
  If the matrices are infinite dimensional, we assume in addition that 
  \begin{equation}
     \label{eq:3.3}
     \left|\sum_{k=0}^\infty \GG{b}_k\GG{s}_k\right|<\infty.
  \end{equation}
  Then $\mathcal{C}=\mathcal{A}\mathcal{B}$ is semi-separable of rank~2 with
  \begin{Eqnarray*}
    \GG{a}_m^{[1]}&=&\GG{d}_m\sum_{k=0}^{m-1}\GG{e}_k\GG{p}_k+\GG{f}_m\GG{p}_m-\GG{a}_m\sum_{k=0}^m \GG{b}_k\GG{p}_k,\qquad 
    \GG{b}_n^{[1]}=\GG{q}_n,\\
    \GG{a}_m^{[2]}&=&\GG{a}_m,\qquad
    \GG{b}_n^{[2]}=\GG{q}_n\sum_{k=0}^{n-1}\GG{b}_k\GG{p}_k+\GG{b}_n\GG{r}_n+\GG{t}_n\sum_{k=n+1}^\infty \GG{b}_k\GG{s}_k,\\[4pt]
    \GG{c}_n&=&\GG{d}_n\GG{q}_n\sum_{k=0}^{n-1}\GG{e}_k\GG{p}_k+\GG{f}_n\GG{r}_n+\GG{a}_n\GG{t}_n\sum_{k=n+1}^\infty \GG{b}_k\GG{s}_k,\\[4pt]
    \GG{d}_m^{[1]}&=&\GG{d}_m,\qquad \GG{e}_n^{[1]}=\GG{q}_n\sum_{k=0}^{n-1}\GG{e}_k\GG{p}_k+\GG{e}_n\GG{r}_n-\GG{t}_n\sum_{k=0}^n \GG{e}_k\GG{s}_k,\\
    \GG{d}_m^{[2]}&=&\GG{d}_m\sum_{k=0}^{m-1}\GG{e}_k\GG{s}_k+\GG{f}_m\GG{s}_m+\GG{a}_m\sum_{k=m+1}^\infty \GG{b}_k\GG{s}_k,\qquad \GG{e}_n^{[2]}=\GG{t}_n, \qquad m,n\in\mathbb{Z}_+.
  \end{Eqnarray*}
\end{lemma}

\begin{proof}
  By direct calculation. Thus, for $m\leq n-1$,
  \begin{Eqnarray*}
  \mathcal{C}_{m,n}&=&\sum_{k=0}^\infty \mathcal{A}_{m,k}\mathcal{B}_{k,n}=\sum_{k=0}^{m-1}\mathcal{A}_{m,k}\mathcal{B}_{k,n}+\mathcal{A}_{m,m}\mathcal{B}_{m,n}+\sum_{k=m+1}^{n-1}\mathcal{A}_{m,k}\mathcal{B}_{k,n}\\
  &&\mbox{}+\mathcal{A}_{m,n}\mathcal{B}_{n,n}+\sum_{k=n+1}^\infty \mathcal{A}_{m,k}\mathcal{B}_{k,n}\\
  &=&\sum_{k=0}^{m-1}\GG{d}_m\GG{e}_k\GG{p}_k\GG{q}_n+\GG{f}_m\GG{p}_m\GG{q}_n+\GG{a}_m\GG{q}_n\!\left(\sum_{k=0}^{n-1} \GG{b}_k\GG{p}_k-\sum_{k=0}^m  \GG{b}_k\GG{p}_k\right)\!+\GG{a}_m\GG{b}_n\GG{r}_n\\
  &&\mbox{}+\sum_{k=n+1}^\infty \GG{a}_m\GG{b}_k\GG{p}_k\GG{q}_n\\
  &=&\!\left(\GG{d}_m\sum_{k=0}^{m-1}\GG{e}_k\GG{p}_k+\GG{f}_m\GG{p}_m-\GG{a}_m\sum_{k=0}^m \GG{b}_k\GG{p}_k\right)\!\GG{q}_n \\
  &&\mbox{}+\GG{a}_m\!\left(\GG{q}_n\sum_{k=0}^{n-1}\GG{b}_k\GG{p}_k +\GG{b}_n\GG{r}_n+\GG{q}_n\sum_{k=n+1}^\infty \GG{b}_k\GG{p}_k\right)=\sum_{i=1}^2 \GG{a}_m^{[i]}\GG{b}^{[i]}_n.
\end{Eqnarray*}
Similar computation applies to the cases $m=n$ and $m\geq n+1$.
\end{proof}

\begin{theorem}
  \label{thm:3.3}
  Let $\mathcal{A}$ and $\mathcal{B}$ be semi-separable matrices of ranks $s$ and $r$ respectively. Then, assuming that the product is bounded, $\mathcal{A}\mathcal{B}$ is semi-separable of rank at most $r+s$.
\end{theorem}

\begin{proof}
  Consistently with \R{eq:3.1}, we  partition $\mathcal{A}$ and $\mathcal{B}$ in the form
  \begin{displaymath}
    \mathcal{A}=\sum_{k=1}^s \mathcal{A}^{[k]},\qquad \mathcal{B}=\sum_{\ell=1}^r \mathcal{B}^{[\ell]},
  \end{displaymath}
  where both the $\mathcal{A}_k$s and the $\mathcal{B}^{[\ell]}$s are all semi-separable of rank 1. We write them in the form \R{eq:3.2} with added superscripts, whereby
  \begin{displaymath}
    \mathcal{A}\mathcal{B}=\sum_{k=1}^s \sum_{\ell=1}^r \mathcal{A}^{[k]}\mathcal{B}^{[\ell]}.
  \end{displaymath}
  Letting $m\leq n-1$ we use Lemma~\ref{lemma:3.2},
  \begin{displaymath}
    (\mathcal{A}\mathcal{B})_{m,n}=\sum_{k=1}^s \sum_{\ell=1}^r \!\left(\GG{a}_m^{[k,1]}\GG{b}_n^{[\ell,1]}+\GG{a}_m^{[k,2]}\GG{b}_n^{[\ell,2]}\right)\!.
  \end{displaymath}
  Since $\GG{b}_n^{[\ell,1]}=\GG{q}_n^{[\ell]}$, $\GG{a}_m^{[k,2]}=\GG{a}_m^{[k]}$, we deduce that
  \begin{displaymath}
    (\mathcal{A}\mathcal{B})_{m,n}=\sum_{\ell=1}^r\!\left(\sum_{k=1}^s \GG{a}_m^{[k,1]}\!\right)\!\GG{q}_n^{[\ell]}+\sum_{k=1}^s \GG{a}_m^{[k]}\!\left(\sum_{\ell=1}^r \GG{b}_n^{[\ell,2]}\!\right)\!.
  \end{displaymath}
  The calculation for $m\geq n+1$ is identical and the theorem follows.
\end{proof}

The boundedness condition can be easily formulated explicitly, {\em \'a la\/} \R{eq:3.3}.

Theorem~\ref{thm:3.3} is somewhat surprising, because a naive `generalisation' of Lemma~\ref{lemma:3.2} seems to imply that the rank of semi-separability propagates multiplicatively, rather than additively, leading to a fairly rapid loss of the benefits of semi-separability. Fortunately, exploiting structure like in the above proof means that this is not the case. Specifically, in our case we can take high powers of $\DDD$ (corresponding to high derivatives) while enjoying progressively diminished -- yet tangible -- benefits of semi-separability.

Note that, once computing finite truncations of infinite-dimensional matrices (like is the case with spectral methods based upon W-functions), it often pays to compute $\sum_{k=n+1}^\infty \GG{b}_k\GG{p}_k$ (and  corresponding sums for higher semi-separability ranks) either explicitly or truncating it much further than the underlying matrices, because the infinite sum might converge slowly.

We now return to the main theme of the paper. In our case $\DDD$ is skew-symmetric and semi-separable of rank~2, hence $\DDD^2$, which corresponds to a Laplacian in a discretised setting, is negative semidefinite and semi-separable of rank 4. In the terminology of the proof of the last theorem, using the notation from \R{eq:3.2}, we have
\begin{Eqnarray*}
  &&\GG{d}^{[i]}=\GG{b}^{[i]},\quad \GG{e}^{[i]}=-\GG{a}^{[i]},\quad \GG{f}^{[i]}=\MM{0},\qquad i=1,2,\\
  &&\GG{p}^{[i]}=\GG{a}^{[i]},\quad \GG{q}^{[i]}=\GG{b}^{[i]},\quad \GG{s}^{[i]}=\GG{b}^{[i]},\quad \GG{t}^{[i]}=-\GG{a}^{[i]},\quad \GG{r}^{[i]}=\MM{0}.
\end{Eqnarray*}
The boundedness condition \R{eq:3.3} boils down to 
\begin{equation}
  \label{eq:3.5}
  \sum_{n=0}^\infty {\GG{b}_n^{[1]}}^2,\; \left|\sum_{n=0}^\infty \GG{b}_n^{[1]} \GG{b}_n^{[2]}\right|\!,\;  \sum_{n=0}^\infty {\GG{b}_n^{[2]}}^2<\infty.
\end{equation}

It follows from Theorem~\ref{thm:3.1} that, in terminology of generalised hypergeometric functions,
\begin{Eqnarray*}
  \sum_{n=0}^\infty {\GG{b}_n^{[1]}}^2&=&\tfrac14 \sum_{n=0}^\infty \frac{(2n+\alpha+\beta+1)\G(n+\alpha+1)}{\G(n+\beta+1)\G(n+\alpha+\beta+1)}\\
  &=&\tfrac14 \frac{\G(\alpha+1)}{\G(\beta+1)\G(\alpha+\beta+1)} \sum_{n=0}^\infty  \frac{(2n+\alpha+\beta+1)n!(\alpha+1)_n}{n!(\beta+1)_n(\alpha+\beta+1)_n}\\
  &=&\tfrac14 \frac{(\alpha+\beta+1)\G(\alpha+1)}{\G(\beta+1)\G(\alpha+\beta+1)}\hyper{2}{2}{1,\alpha+1}{\beta+1,\alpha+\beta+1}{1}\\
  &&\mbox{}+\tfrac12 \frac{\G(\alpha+2)}{\G(\beta+2)\G(\alpha+\beta+2)} \hyper{2}{2}{2,\alpha+2}{\beta+2,\alpha+\beta+2}{1}\!;\\
  \sum_{n=0}^\infty {\GG{b}_n^{[2]}}^2&=&\tfrac14 \sum_{n=0}^\infty \frac{(2n+\alpha+\beta+1)\G(n+\beta+1)}{\G(n+\alpha+1)\G(n+\alpha+\beta+1)}\\
  &=&\tfrac14 \frac{(\alpha+\beta+1)\G(\beta+1)}{\G(\alpha+1)\G(\alpha+\beta+1)}\hyper{2}{2}{1,\beta+1}{\alpha+1,\alpha+\beta+1}{1}\\
  &&\mbox{}+\tfrac14 \frac{\G(\beta+2)}{\G(\alpha+2)\G(\alpha+\beta+2)} \hyper{2}{2}{2,\beta+2}{\alpha+2,\alpha+\beta+2}{1}\!;\\
  \sum_{n=0}^\infty \GG{b}_n^{[1]}\GG{b}_n^{[2]}&=&-\tfrac14 \sum_{n=0}^\infty (-1)^n \frac{2n+\alpha+\beta+1}{\G(n+\alpha+\beta+1)}\\
  &=&-\tfrac14 \frac{\alpha+\beta+1}{\G(\alpha+\beta+1)} \hyper{1}{1}{1}{\alpha+\beta+1}{-1}\\
  &&\mbox{}+\tfrac12 \frac{1}{\G(\alpha+\beta+2)} \hyper{1}{1}{2}{\alpha+\beta+2}{-1}\!.
\end{Eqnarray*}
All ${}_qF_q$ hypergeometric functions are entire \cite{rainville60sf} and it thus follows that the boundedness condition \R{eq:3.5} holds.\footnote{The above functions can be summed up using  incomplete Gamma function.}

We deduce that for all $\alpha,\beta>0$ the matrix $\DDD^2$ is bounded and semi-separable of rank 4. Higher powers of $\DDD$ can be calculated in a similar manner.

\section{Reflections on semi-separability and differentiation matrices}

A curious coincidence  might -- or might not -- be an indicator of deeper structure. The Hermite weight $\ee^{-x^2}$, $x\in\mathbb{R}$, yields a tridiagonal differentiation matrix which is `almost' (i.e.\ disregarding the first sub- and super-diagonals) semi-separable of rank~0. Both the  Laguerre weight and the ultraspherical weight produce a semi-separable differentiation matrix of rank~1, while the Jacobi weight leads to rank~2: this corresponds exactly to the number of parameters in the underlying orthogonal polynomials, $\CC{H}_n,\CC{L}_n^{(\alpha)},\CC{P}_n^{(\alpha,\alpha)}$ and $\CC{P}_n^{(\alpha,\beta)}$. This is evocative of the {\em Askey scheme,\/}  classifying orthogonal polynomials that can be expressed as generalised hypergeometric functions by the number of their parameters,  going considerably further than just two parameters \cite{askey85sbh}. Is there a connection between the Askey scheme, which has led to much fruitful insight into orthogonal polynomials, and semi-separability? Pass.

Another issue which might be interesting is the geometry of semi-separable matrices. The condition \R{eq:3.1} is necessary and sufficient for a matrix to be semi-separable of rank $s\in\{0,\ldots,r\}$ -- indeed, \citeasnoun{vandebril08mcs} emphasize the difference between {\em semi-separable matrices\/} (which are exactly of rank $r$) and {\em semi-separably generated matrices\/} defined by \R{eq:3.1}, of rank $\leq r$. Given $N\in\mathbb{N}\cup\{\infty\}$ and $r\in\{0,\ldots,N\}$, $r<\infty$, we let $\MM{S}_{N,r}$ and $\MM{T}_{N,r}$ denote the set of all $N\times N$ matrices which are semi-separable of exactly rank $r$ and semi-separable of rank $\leq r$, respectively.  It takes but a moment reflection to persuade ourselves that the set $\MM{S}_{N,r}$ has poor analytic properties -- in particular, it is not closed in any reasonable topology. However, the $\MM{T}_{N,r}$s are closed in the $\ell_2$ topology, form manifolds and these manifolds form a {\em flag\/}
\begin{displaymath}
  \MM{T}_{N,0}\subset\MM{T}_{N,1}\subset\cdots\subset \MM{T}_{N,r}\subset \cdots\subset\MM{T}_{N,N}.
\end{displaymath}
(Cf.~\cite{humphreys72ila} for a definition of flag.) We have already proved in Theorem~\ref{thm:3.3} that $\MM{T}_{N,r}\times\MM{T}_{N,s}\subseteq \MM{T}_{N,r+s}$ (subject to boundedness conditions for $N=\infty$). Likewise, it is trivial that $\MM{T}_{N,r}+\MM{T}_{N,s}\subseteq \MM{T}_{N,r+s}$. While the implications of this are clear within the subject matter of this paper, the design of spectral methods for partial differential equations, there might be broader implications in the theory of semi-separable matrices.

\bibliographystyle{agsm}


\end{document}